\documentclass[leqno,12pt]{amsart} 
\title{Ternary mappings of some evolution algebras}
\author{C\'andido Mart\'in Gonz\'alez$^{1}$, Jacques Rabie$^{2}$, Juana S\'anchez-Ortega$^{2}$}
\usepackage{amssymb,color, bm}


\usepackage{amssymb, mathrsfs, amsmath, tabu, amsthm}
\usepackage{hyperref}
\usepackage{tikz-cd, graphicx} 
\usepackage{pdfpages}
\usepackage{enumitem}
\usepackage{ragged2e}
\usepackage[utf8]{inputenc}
 \usepackage{longtable}
\usetikzlibrary{decorations.markings}

\hypersetup{
	colorlinks = true,
	urlcolor = blue,
	linkcolor = blue,
	citecolor = red
}

\newcommand{\comment}[1]{}

\def\taffaut{\mathop{\hbox{\bf Taut}}}
\def\taut{\mathop{\hbox{\rm Taut}}}

\def\tder{\mathop{\hbox{\rm Tder}}}
\def\alg{\mathop{\hbox{\bf alg}}}
\def\grp{\mathop{\hbox{\bf grp}}}
\def\Gm{\mathop{\hbox{\bf G}_{\bf m}}}

\def\s{\sigma}
\def\t{\tau}

\newcommand{\K}{\mathbb{K}}
\newcommand{\N}{\mathbb{N}}
 
\newcommand{\A}{\mathcal{A}}

\newcommand{\B}{\mathcal{B}}
\newcommand{\M}{\mathcal{M}}
\newcommand{\D}{\mathcal{D}}

\theoremstyle{plain} 
\newtheorem{theorem}{Theorem}[section]
\newtheorem{lemma}[theorem]{Lemma}
\newtheorem{cor}[theorem]{Corollary}
\newtheorem{prop}[theorem]{Proposition}

\theoremstyle{definition} 
\newtheorem{define}[theorem]{Definition}

\newtheorem{example}[theorem]{Example}

\begin{document}


\address[1]{Departamento de \'Algebra, Geometr\'ia y Topolog\'ia. Universidad de M\'alaga. Spain}
\address[2]{Department of Mathematics and Applied Mathematics. University of Cape Town. South Africa}

\email{candido\_\,m@uma.es}
\email{jacques.rabie@uct.ca.za}
\email{juana.sanchez-ortega@uct.ac.za}

\maketitle

\begin{abstract}
The group scheme of ternary automorphisms of a perfect finite dimensional evolution algebra 
$\A$ is computed. The main advantage of using group schemes is that it allows to apply the Lie functor to determine the Lie algebra of ternary derivations of $\A$. Using the generalised inverse of a matrix, we provide a precise classification of all ternary derivations of an arbitrary finite-dimensional evolution algebra $\A$. The ternary derivations of all 2-dimensional evolution algebras are also computed. 
\end{abstract}

 \section{Introduction}  
 
 The notion of an evolution algebra was introduced in 2006 by Tian and Vojtechovsky \cite{TV} to model algebraically the self reproduction process in non-Mendelian genetics. The theory of evolution algebras has many applications and connection to other areas of Mathematics and other fields, like, for instance, graph theory, group theory, Markov chains, Mathematical Physics, Genetics as shown in Tian's book \cite{Tbook}. Due to this fact, evolution algebras have been objects of interests of several researchers. For instance, different types of evolution algebras have been classified and studied in \cite{AA0, CCSMVV, squares, CV}; derivations and automorphisms of certain evolution algebras have been investigated in \cite{AA, CCCR, CCCRR, CGOT}. We refer the reader to \cite{CFNVT} for a survey on the recent developments on the theory of evolution algebras. 
 
In this paper, we focus our attention on ternary automorphisms and ternary derivations of finite-dimensional evolution algebras. Ternary derivations are triple of linear maps satisfying a certain derivation identity. More precisely, a triple of linear maps $(d_1, d_2, d_3)$ on an algebra $\mathcal{A}$ is said to be a {\bf ternary derivation} of $\mathcal{A}$ if it satisfies the following: 
\begin{equation}
   d_1(xy) = d_2(x)y + x d_3(y),  \, \, \mbox{ for all } \, \, x, y \in \mathcal{A}. \tag{Td} \label{Td}
\end{equation}
We could say that ternary derivations originated with Newton and Leibniz but the general notion was introduced later on in 
\cite{JGPI}. In \cite{JGPI2, PI} ternary derivations were used to study some unital associative algebras; in \cite{S} ternary derivations of separable associative and Jordan algebras were described, while ternary derivations and ternary automorphisms of triangular algebras were investigated in \cite{MBMGSOV}. The notion of a ternary automorphisms was introduced in \cite{JGPI}. A triple of bijective linear maps $(f_1, f_2, f_3)$ on $\A$ is called a {\bf ternary automorphism} of $\A$ if the identity below is satisfied for all $x, y \in \A$. 
\begin{equation}
f_1(xy) = f_2(x)f_3(y). \tag{Ta} \label{Ta} 
\end{equation}
The set $\taut(\A)$ consisting of all ternary automorphisms of $\A$ has a group structure with the component-wise composition; while the set $\tder(\A)$ consisting of all ternary derivations of $\A$ has a Lie algebra structure with the component-wise Lie bracket.

Our motivation comes from the results obtained for automorphisms and derivations of evolution algebras and \cite[Problem 4.1]{CFNVT}. We solve such a problem in this paper. To be more precise, in \S\ref{sec2} we introduce the necessary background and prove some preliminary results on evolution algebras, polynomial ideals and group schemes. 
Using the theory of group schemes, we investigate the ternary automorphisms of a perfect evolution algebra $\A$ in \S\ref{sec3}. A necessary and sufficient condition for a triple of bijective linear maps $(f_1, f_2, f_3)$ to be a ternary automorphism is provided in Lemma \ref{necsufTA}. Moreover, we prove that $f_1$ is completely determined by $f_2$ and $f_3$. Once the group scheme of ternary automorphism is
known, we describe $\tder(\A)$ (see Theorem \ref{TAUT} and Corollary \ref{TderPer}). At this point, 
it is worth mentioning here that imposing the evolution algebra to be perfect should not surprise the 
reader, since automorphisms have been fully described, only in the context 
of perfect evolution algebras (see \cite{AA} for more details). Regarding to ternary derivations, using linear algebra techniques, we are able to remove the perfection hypothesis on $\A$. More concretely, in \S\ref{sec4} we provide necessary and sufficient conditions for a triple of linear maps $(d_1, d_2, d_3)$ to be a ternary derivation (see Proposition \ref{iffCondition}). The first condition relates $d_2$ and $d_3$, while the second condition links $d_1$ with the diagonal of $d_2 + d_3$. These two conditions are also expressed as a linear system of equations (see \eqref{system1}) and as a matrix equation (see \eqref{system3}).
In the perfect case, we improve the result obtained in the previous section providing a precise description (see Corollary \ref{iffTD}) of all ternary derivations, including a formula to compute $d_1$. In the non-perfect case, we determine necessary conditions for $d_2$ and $d_3$ to satisfy \eqref{system1} (see Proposition \ref{d2d3}). Next, using generalised matrix inverses, we are able to find the general solution of \eqref{system3} for $d_1$ (see Lemma \ref{partialsol}), as well as, necessary and sufficient conditions for \eqref{system3} to have a solution (see Proposition \ref{d1solution}). We collect all these results in Theorem \ref{charDerTarb}, which provides a precise characterisation of ternary derivations of (arbitrary) finite-dimensional evolution algebras.
We close the paper computing all the ternary derivations of 2-dimensional evolution algebras in \S\ref{case2}. 

	
\section{Preliminaries}  \label{sec2}

In this paper, we consider algebras of finite dimension over an arbitrary field 
$\K$ of any characteristic. Recall that by an algebra $\A$ we mean a vector space endowed with a bilinear operation 
$\A \times \A \to \A$, denoted $(a, b) \mapsto ab$, and called the product of $\A$. 

\subsection{Evolution algebras} 
Let $n \geq 1$ and write $\Lambda$ to denote the set $\{1, \ldots, n\}$. An $n$-dimensional algebra $\A$ is called an {\bf evolution algebra} if it it admits a basis $\B = \{e_i\}_{i \in \Lambda}$, called a {\bf natural basis} of $\A$, such that 
\begin{align}
e_i \cdot e_i &= \sum^n_{k = 1} w_{ki}e_k, \quad \mbox{for} \quad i \, \in \Lambda, \tag{E1} \label{E1}
\\
e_i \cdot e_j &= 0, \, \, \mbox{for} \, \, i, j \, \in \Lambda \, \, \mbox{with} \, \,  i \neq j \tag{E2}. \label{E2}
\end{align}
The scalars $w_{ki}$ from \eqref{E1} are called the {\bf structure constants} of $\A$ with respect to $\B$, while the matrix $\M_\B:= (w_{ki})$ is often called the {\bf structure matrix} of $\A$. 
If $\A^2 = \A$, or equivalently $\M_\B$ is invertible, we say that $\A$ is {\bf perfect}.

\subsection{Polynomial ideals}
Let $D$ be a domain, $R = D[x,\, y,\, z]$ the ring of polynomials in the indeterminates $x,\, y, \, z$ and the ideal $I =  (xyz - 1)$ of $R$. We claim that the factor ring $R/I$ is isomorphic to the localization $D[x,y]S^{-1}$, where $S =\{(xy)^n \mid n \in\N\}$ (here $0 \in \N$). In fact, using that the map $\pi: D[x,y,z]\to D[x,y]S^{-1}$ given by $p(x, y, z)\mapsto p\left(x,y,\frac{1}{xy}\right)$ is a ring epimorphism, one can easily prove that any ring homomorphism $f: D[x, y, z]\to S$ such that $f(I) = 0$ factorizes through $\pi$. This proves that $D[x, y]S^{-1}\cong R/I$, which implies that $I$ is a prime ideal.
More generally, for $R = \K[x_1,\ldots, x_n, z]$ one can prove that $I = \left(z\prod_1^n x_i - 1 \right)$ is a prime ideal of $R$, which yields that $R/I$ is a domain. To do so, use that $R/I \cong RS^{-1}$ for $S = \{(\prod_1^n x_i)^n \mid n \in\N\}$.

Given $k > 1$, we claim that the ideal $\left( z \prod_{i=1}^k x_i-1, \, w \prod_{i=1}^k y_i-1 \right)$ of the polynomial algebra $\K[x_1, \ldots x_k,\, y_1, \ldots, y_k, \, z,\, w]$ is prime. We first prove a preliminary lemma:

\begin{lemma}\label{unouno}
Let $X$ and $Y$ be disjoint sets of indeterminates and $I\triangleleft \K[X]$, $J\triangleleft \K[Y]$.
Then, the following isomorphisms hold for the ideal $I^e$ (respectively, $J^e$) of $\K[X\sqcup Y]$ generated by the image of $I$ (respectively, $J$). 
\[
\K[X]/I\otimes \K[Y]/J\cong (\K[X]\otimes \K[Y])/(I\otimes \K[y] + \K[x]\otimes J)\cong \K[X\sqcup Y]/(I^e+J^e).
\] 
\end{lemma}
 
\begin{proof} 
Apply \cite[Theorem 7.7]{Hu} to the projections $\K[X]\to \K[X]/I$, $\K[Y]\to \K[Y]/J$ to get the first isomorphism. For the second, apply the First Isomorphism Theorem to the epimorphism $\K[X]\otimes \K[Y]\to \K[X\sqcup Y]/(I^e+J^e)$ mapping the element $\sum_i (p_i(X)\otimes q_i(Y))$ onto $\sum_i \overline{p_i(X)}\overline{q_i(Y)}$, where $\overline{a}$ denotes the equivalence class of $a$. 
\end{proof}

\begin{prop} \label{unodos} 
For $k > 1$, the ideal $I:= \left( z \prod_{i=1}^k x_i-1, \, w \prod_{i=1}^k y_i-1 \right)$ of the polynomial algebra $R:= \K[x_1, \ldots x_k,\, y_1, \ldots, y_k, \, z,\, w]$ is prime. 
\end{prop}

\begin{proof} 
Lemma \ref{unouno} yields that $R/I \cong A \otimes B$, where 
$A:= \K[x_1, \ldots, x_k, z]/\big(z\prod x_i-1\big)$ and $B:= \K[y_1, \ldots, y_k, \,w]/\big(w\prod y_i - 1\big)$ are both domains by the considerations above. Moreover, their scalar extensions $A_{\overline{\K}}$ and $B_{\overline{\K}}$, respectively, (where $\overline{\K}$ denotes the algebraic closure of $\K$) are finitely generated domains. Keeping in mind that the tensor product of finitely generated domains over algebraically closed fields is also a domain, we obtain that 
$A_{\overline{\K}} \otimes_{\overline{\K}} B_{\overline{\K}}$ is a domain. From here, using that $(A\otimes B)_{\overline{K}}\cong A_{\overline{\K}}\otimes_{\overline{\K}} B_{\overline{\K}}$ (see \cite[Exercise 2.15]{atiyah}), we obtain that 
$A\otimes B$ is a domain, as required.
\end{proof}

\subsection{Group schemes. The algebraic torus}
In this paper, we will make use of the affine $\K$-group scheme $\Gm$, called the {\bf multiplicative group scheme}; which is represented by the Hopf algebra $\K[x^\pm]$, that is, the Laurent polynomial algebra.
For $k\in\N\setminus\{0\}$, the product group $\Gm^k$ will be called the {\bf $k$-dimensional torus}.
For a finite group $G$, we write $\underline{G}$ to refer to the induced {\bf constant group scheme} for 
$G$ (see \cite[\S2.3, p. 16]{waterhouse}). If $S_n$ denotes the permutation group of 
order $n!$, then the corresponding $\K$-group scheme is $\underline{S_n}$ whose representing Hopf algebra is 
$\K^{n!}$ endowed with a Hopf algebra structure as in \cite[\S2.3, p. 16]{waterhouse}.


\section{Ternary automorphisms on perfect evolution algebras}   \label{sec3}

In this section, $\A$ denotes an $n$-dimensional perfect evolution algebra. We will use the same notation as in \S2.1 for its natural base and its structure matrix. For a triple $(f_1, f_2, f_3)$ of linear bijections on $\A$ we write
\begin{equation} \label{TautP}
f_1(e_i) = \sum_{k \in \Lambda} c_{ki} e_k, 
\quad
f_2(e_i) = \sum_{k \in \Lambda} a_{ki} e_k, 
\quad
f_3(e_i) = \sum_{k \in \Lambda} b_{ki} e_k, 
\end{equation}
for $c_{ki}, a_{ki}, b_{ki} \in \K$. We begin by providing a necessary and sufficient condition for $(f_1, f_2, f_3) \in \taut(\A)$. 

\begin{lemma} \label{necsufTA}
A triple $(f_1, f_2, f_3)$ of bijective linear maps on $\A$ is in $\taut(\A)$ if and only if it satisfies the identity below for  all $i, j, k \in \Lambda$ with $i \neq j$. 
\begin{equation} \label{wssap}
a_{ki} b_{kj} = 0.
\end{equation}
Moreover, $f_1$ is completely determined by $f_2$ and $f_3$.
\end{lemma}

\begin{proof}
Let us begin by noticing that $\{e^2_i\}$ is also a basis of $\A$ (as a vector space over $\K$). Using this, \eqref{wssap} can be proved by applying \eqref{Ta} and \eqref{E1}.
To finish, notice that 
\[
f_1(e_i^2) = f_2(e_i)f_3(e_i) = \left(\sum_{k \in \Lambda} a_{ki} e_k \right)\left(\sum_{\ell \in \Lambda} b_{\ell i} e_\ell \right) = \sum_{k \in \Lambda} a_{ki}b_{ki} e_k,
\]
finishing the proof.
\end{proof}
\subsection{The group scheme of a finite-dimensional evolution algebra}

We write $\alg_\K$ to denote the category of all associative and commutative unital $\K$-algebras, while 
$\grp$ refers to the category of groups. The {\bf group functor of ternary automorphisms} 
of an $n$-dimensional evolution algebra $\A$ is defined as  
$\taffaut(\A): \alg_\K \to \grp$, where $\taffaut(\A)(R):= \taut(\A_R)$ for $\A_R$ the scalar extension 
of $\A$. If $A$ is perfect, taking into account Lemma \ref{necsufTA}, the Hopf algebra representing this group scheme is the 
quotient polynomial algebra 
$\K[x_i^k, \, y_i^k,\, z,\, w]/I$, where $A:=\K[x_i^k,\, y_i^k,\, z,\, w]$ is the polynomial algebra in the indeterminates $\{z, w\}\cup\{x_i^k\}_{i, k \in \Lambda} \cup \{y_i^k\}_{i, k \in \Lambda}$ (where $\Lambda:= \{1, \ldots, n\}$) and $I$ the following ideal:
\begin{equation} \label{Iideal} 
I = \Big( \{x^k_iy^k_j\}_{i \neq j} \cup \{z\det(x_i^j) - 1, \, w\det(y_i^j) - 1\}_{i, j \in \Lambda} \Big).
\end{equation}

Following \cite[\S1.4, p. 7]{Jantzen} we identify $\taffaut(\A)$ with $V(I)$ (see notations therein) so that $V(I)(\K)$ is the zero locus of the ideal $I$. 

For $\sigma \in S_n$ we consider the ideal 
\begin{equation} \label{Jsigma}
J_\sigma = \left(\{x_i^j\}_{j\ne\s(i)} \cup \{y_i^j\}_{j\ne\s(i)} \cup \left\{(-1)^{\vert\s\vert}z\prod_i x_i^{\s(i)}-1, \,(-1)^{\vert\s\vert}w\prod_i y_i^{\s(i)}-1 \right\}\right).
\end{equation} 
It is straightforward to show that $J_\sigma$ is a proper ideal of $A$ for all $\sigma \in S_n$.

\begin{example}
For $n = 2$, the Hopf algebra is $A:= \K[x_1^1,\, x_1^2, \, x_2^1,\, x_2^2,\, y_1^1,\, y_1^2,\, y_2^1,\, y_2^2,\, w,\, z]$ and $I$ is the ideal generated by the following polynomials:
\[
x_1^1y_2^1, \, \, x_2^1y_1^1, \, \, x_1^2y_2^2, \, \, x_2^2y_1^2, \, \, z(x_1^1x_2^2-x_1^2x_2^1) - 1, \, \, w(y_1^1y_2^2-y_1^2y_2^1) - 1.
\]
In this case, we only have two ideals, namely:
\[
J_1 := \Big( x_1^2, \, x_2^1, \, y_1^2, \, y_2^1, \, z x_1^1 x_2^2 - 1, \, w y_1^1y_2^2 - 1 \Big), 
\, \, 
J_2 := \Big( x_1^1, \, x_2^2, \, y_1^1, \, y_2^2, \, z x_1^2 x_2^1 + 1, \, w y_1^2y_2^1 + 1 \Big).
\]
We claim that $J_1$ and $J_2$ are the unique prime ideals containing $I$. In fact, suppose that $P$ is a prime ideal containing $I$. Then either $x_1^1\in P$ or $y_2^1\in P$.
\begin{itemize}
\item[-] If $x_1^1\notin P$, then $y_2^1 \in P$. In particular, $wy_1^1y_2^2 - 1 \in P$, which implies that $y_1^1, y_2^2 \notin P$ and so $y_1^1y_2^2 \notin P$. Moreover, $x_2^1, \, x_1^2 \notin P$, which yields $y_1^2\in P$ and $x_1^1x_2^2\notin P$. Thus $\tiny\begin{pmatrix}y_1^1 & y_1^2\cr y_2^1 & y_2^2\end{pmatrix}\equiv\begin{pmatrix}y_1^1 & 0\cr 0 & y_2^2\end{pmatrix} \pmod{P}$ and $y_1^1y_2^2\notin P$ and $\tiny\begin{pmatrix}x_1^1 & x_1^2\cr x_2^1 &x_2^2\end{pmatrix}\equiv\begin{pmatrix}x_1^1 & 0\cr 0 & y_2^2\end{pmatrix} \pmod{P}$.
\item[-] If $x_1^1\in P$, then $x_1^2, x_2^1\notin P$ and so $x_1^2x_2^1 \notin P$. Moreover, $y_1^1, y_2^2 \in P$, which implies that $y_1^2, y_2^1 \notin P$ and so $y_1^2 y_2^1 \notin P$. From here, we obtain that $x_2^2\in P$. Thus
$\tiny\begin{pmatrix}y_1^1 & y_1^2\cr y_2^1 & y_2^2\end{pmatrix}\equiv\begin{pmatrix}0 & y_1^2\cr y_2^1 & 0\end{pmatrix} \pmod{P}$ and $\tiny\begin{pmatrix}x_1^1 & x_1^2\cr x_2^1 & x_2^2\end{pmatrix}\equiv\begin{pmatrix}0 & x_1^2\cr x_2^1 & 0\end{pmatrix} \pmod{P}$.
\end{itemize} 
To finish, notice that both $J_1$ and $J_2$ are prime by Proposition \ref{unodos}. We claim that $I = J_1 \cap J_2$; in fact, $J_1 + J_2 = A$, since $wy_1^1y_2^2 - 1 \in J_1$ and $wy_1^2y_2^1 + 1\in J_2$. Lastly, one can easily proved that $J_1J_2\subset I$, which imply that $I \subseteq J_1\cap J_2 \subseteq J_1J_2  \subseteq I$, as required.  

\smallskip

Thus, $I = J_1 \cap J_2$, and so scheme theoretically $V(I) = V(J_1)\cup V(J_2)$ 
(see \cite[p. 8]{Jantzen}). The closed subfunctors $V(J_i)$ are the irreducible connected component of $V(I)$. 
In fact, the torus $V(J_1)\cong(\Gm^2)^2$ is the component of the unit and $V(I)/V(J_1)$ the group of components which is isomorphic to the constant finite group associated to $S_2$ the group of permutations two elements.

Summarizing, for a $2$-dimensional perfect evolution algebra $\A$ we have a short exact sequence: 
$(\Gm^2)^2\hookrightarrow\taffaut(\A)\twoheadrightarrow\underline{S_2}.$
Moreover, the Hopf algebra of $\taffaut(\A)$ is the tensor 
product of those of $(\Gm^2)^2$ and $\underline{S_2}$, that is $\K[x^\pm]^{\otimes 4}\otimes K^2$. 
This point deserves a little explanation. The representing Hopf algebra of $\taffaut(\A)$ is 
$$A/I\cong A/J_1\times A/J_2\cong (A/J_1)\otimes \K^2,$$ where the last isomorphism is given by 
$(a + J_1, b + J_2)\mapsto (a + J_1)\otimes u_1+\Omega(b + J_2)\otimes u_2$, for $\{u_1, u_2\}$ the canonical basis of $\K^2$,  
and $\Omega\colon A/J_2\to A/J_1$ the isomorphism induced by the automorphism $A\to A$ such that 
$x_i^j\mapsto x_i^{\s(j)}$, $y_i^j\mapsto y_i^{\s(j)}$,
$z\mapsto -z$, $w\mapsto -w$ (where $\sigma\in S_2$ denotes the transposition $(1 \quad 2)$).

So 
$\taffaut(\A)\cong(\Gm^2)^2\times \underline{S_2}$. Since the Lie functor preserves (fibered) products (see \cite[Chapter 10, \S c, p. 190]{Milne}) we get $\tder(\A)\cong \K^2\times \K^2$. This implies that the ternary derivations are triples $(d_1,d_2,d_3)$ such that in a natural basis $\{e_i\}$ of $\A$, the maps $d_2$ and $d_3$ are diagonal and $d_1(e_i^2)=d_2(e_i)e_i+e_id_3(e_i)$ (so $d_1$ is completely determined by $d_2$ and $d_3$). As a corollary $\tder(A)\cong \D_2(\K)\times \D_2(\K)$, where
$\D_n(\K)$ denotes the space of diagonal $n\times n$ matrices with coefficients in $\K$.
\end{example}

The following result generalises the example above.

\begin{lemma}\label{emirp}
Let $\sigma, \tau \in S_n$ be distinct. Then the following hold:
\begin{itemize}
\item[\rm (i)] $J_\s$, defined as in \eqref{Jsigma}, is prime and contains $I$;
\item[\rm (ii)] $J_\s$ and $J_\t$ are coprime;
\item[\rm (iii)] $J_\s = \Big( \{x_i^j\}_{j\ne\s(i)} \cup  \{y_i^j\}_{j\ne\s(i)} \cup \{z\det(x_i^j)-1, \, w\det(y_i^j)-1\}_{i, j \in \Lambda}\Big)$.
\end{itemize}
\end{lemma}

\begin{proof}
(i) We begin by proving that $I \subseteq J_\sigma$. Notice that any element of the form $x^k_{\s(i)}y^k_j$, for $j \neq i$ is in $J_\s$; in fact, if $k = \s(i)$, then $k \neq \s(j)$ and so $y^{k}_j \in J_\sigma$. We prove here that $z\det(x_i^j) - 1 \in J_\sigma$; the proof of $w\det(y_i^j) - 1 \in J_\s$ is similar and we leave it to the reader. To do so, notice that 
\[
z\det(x_i^j)-1=
\left((-1)^{\vert\s\vert} z\prod_i x_i^{\s(i)} - 1\right) + \sum_{\tau\ne\s}(-1)^{\vert\tau\vert}z\prod_i x_i^{\tau(i)}.
\]
It remains to show that $\sum_{\tau\ne\s}(-1)^{\vert\tau\vert}z\prod_i x_i^{\tau(i)} \in J_\sigma$. This follows from $\tau\ne\s$. In fact, since $\tau\ne\s$ there exists $i$ such that $\tau(i)\ne\s(i)$ and so $x_i^{\tau(i)}\in J_\s$.

\smallskip

Notice that it is enough to prove that $J_1$ is prime; in fact, any $\s\in S_n$ induces an automorphism of $A$ fixing $z$ and $w$ and such that $x_i^k\mapsto x_{\s(i)}^{\s(k)}$ and $y_i^k\mapsto y_{\s(i)}^{\s(k)}$. Such automorphisms permute the ideals $J_\sigma$. To prove that $J_1$ is prime we will show that $A/J_1$ is a domain, which follows from Lemma \ref{unodos}: 
\[
A/J_1 \cong K[x_{11}, \ldots, x_{nn}, \, y_{11},\ldots, y_{nn},\, z, \, w]/\big(z\prod x_{ii} - 1, \, w\prod y_{ii} - 1\big).
\]

\noindent (ii) Consider the subsets $S = \{k \in \Lambda \mid \s(k) = \t(k)\}$ and $T := \Lambda \setminus S$. Notice that $T \neq \emptyset$ since we are assuming that $\s \neq \t$. Using that 
$(-1)^{\vert\s\vert}z\prod_{k\in A} x_k^{\s(k)}\prod_{k\in B} x_k^{\s(k)} - 1 \in J_\sigma$ 
and $\prod_{k\in B} x_k^{\s(k)} \in J_\tau$ we obtain that $1 \in  J_\s + J_\t$, as required.

\smallskip

\noindent (iii) Write $z\det(x_i^j) - 1 = z\sum_\t d_\t - 1$, for $d_\t = (-1)^{\vert\t\vert}\prod_i x_i^{\t(i)}$. Notice that 
$d_\t \in J_\s$ for all $\t \neq \s$, since in such a case $\s(\ell) \neq \t(\ell)$ for some $\ell$ and so the factor $x_\ell^{\t(\ell)} \in J_\s$. Thus: $z\det(x_i^j) - 1 \equiv z d_\s-1 = (-1)^{\vert\s\vert}z\prod_i x_i^{\s(i)} \pmod{K_\sigma}$, for $K_\sigma = \langle \{x_i^j\}_{j\ne\s(i)}, \{y_i^j\}_{j\ne\s(i)} \rangle$. Similarly, one can prove that $w\det(y_i^j)-1\equiv (-1)^{\vert\s\vert}w\prod_i y_i^{\s(i)} \pmod{K_\sigma},$ which finishes the proof.
\end{proof}

\begin{lemma} \label{lemaunido}
For any $\t \in S_n$ choose an indeterminate $x_{i_\t}^{j_\t}$ such that $j_\t\ne \t(i_\t)$. Then the following assertions hold.
\begin{itemize}
\item[\rm (i)] $y_1^{\s(1)}\ldots y_n^{\s(n)}\prod_{\t\in S_n} x_{i_\t}^{j_\t}\in I$ for all $\sigma \in S_n;$
\item[\rm (ii)] $\prod_{\t\in S_n} x_{i_\t}^{j_\t} \in I;$
\item[\rm (iii)] $x_i^j y_1^{\s(1)}\cdots y_n^{\s(n)}\in I$ for all $\s \in S_n$ and $j \neq \s(i)$.
\end{itemize}
\end{lemma}

\begin{proof}
(i) Notice that $y_1^{\s(1)}\ldots y_n^{\s(n)}\prod_{\t\in S_n} x_{i_t}^{j_t} =
y_1^{\s(1)}\ldots y_n^{\s(n)}x_{i_\s}^{j_\s}\prod_{\t\ne\s} x_{i_t}^{j_t}$. Then $x_{i_\s}^{j_\s}y_{i_\s}^{\s(i_\s)}$
is in $I$, since $j_\s\ne \s(i_\s)$ by hypothesis, and (i) follows.

\smallskip

\noindent (ii) follows from (i) since $\left(w\det(y_i^j)-1\right)\prod_{\t\in S_n} x_{i_t}^{j_t} \in I.$

\smallskip

\noindent (iii) Let $k \in \Lambda$ such that $j = \s(k)$. 
Then $x_i^j y_k^{\s(k)} \in I$, since $j=\s(k)$ and $i\ne k$, which implies that $x_i^j y_k^{\s(k)} \prod_{q\ne k}y_q^{\s(q)}\in I$. 
\end{proof}

\begin{lemma}\label{oetse}
Let $S_n = S \sqcup T$ be a partition such that both $S$ and $T$ are non-empty. Then the element $v: = \prod_{\t\in S}x_{i_\t}^{j_\t}\prod_{\s\in T}y_{k_\s}^{q_\s}$, where $j_\t\ne\t(i_\t)$ and $q_\s\ne\s(k_\s)$ for all  
$\t\in S$ and $\s\in T$, belongs to $I$. 
\end{lemma}

\begin{proof}
Let $\rho\in S_n$ and $v_\rho:= v y_1^{\rho(1)} \ldots y_n^{\rho(n)}$. If $\rho\in S$ then $\xi:=x_{i_\rho}^{j_\rho}y_1^{\rho(1)}\cdots y_n^{\rho(n)}$ is a factor of $v_\rho$ and belongs to $I$ by Lemma \ref{lemaunido} (iii). From here we obtain that $v_\rho \in I$. Similarly, one can show that $v x_1^{\rho(1)} \ldots x_n^{\rho(n)} \in I$, provided that $\rho \in T$. On the other hand, we have that $\det(x_i^j) = x_S + x_T$, 
where $x_S = \sum_{\s\in S}(-1)^{\vert\s\vert}x_1^{\s(1)}\cdots x_n^{\s(n)}$ and $x_T = \det(x_i^j)- x_S$. Similarly, $\det(y_i^j) = y_S + y_T$, where $y_S = \sum_{\s\in S}(-1)^{\vert\s\vert}x_1^{\s(1)}\cdots x_n^{\s(n)}$ and $y_T = \det(y_i^j)-y_S$. We have proved that $vy_S, vx_T \in I$, which yield that 
\[
v\det(x_i^j)\det(y_i^j) = v x_Sy_S + vx_Sy_T + vx_Ty_S + vx_Ty_T \equiv vx_Sy_T \pmod{I}
\]
We focus our argument on $x_Sy_T$, which is a linear combinations of
elements of the form $x_1^{\s(1)}\cdots x_n^{\s(n)}y_1^{\t(1)}\cdots y_n^{\t(n)}$ with $\s\in S$ and $\t\in T$. We claim 
that $x_1^{\s(1)}\cdots x_n^{\s(n)}y_1^{\t(1)}\cdots y_n^{\t(n)}$ are all in $I$. In fact, consider $x_1^{\s(1)}$. Then there is some $j$ such that $\s(1)=\t(j)$. If $j\ne 1$, then $x_1^{\s(1)}y_j^{\t(j)}\in I$ and we are done. Otherwise $\s(1) = \t(1)$. Consider now $x_2^{\s(2)}$; there must be some $j$ such that $\s(2)=\t(j)$. Thus $x_2^{\s(2)}y_j^{\t(j)}\in I$ unless $j = 2$. In this case we have $\s(2)=\t(2)$. Proceeding in this way, given that $\s\ne\t$ we can find some $i$ such that
$\s(k)=\t(k)$ for $1\le k\le i-1$ and $\s(i)\ne \t(i)$. Again
$\s(i)=t(j)$ for some $j$. Thus $x_i^{\s(i)}y_j^{\t(j)}\in I$ unless $i=j$, which is impossible since this would imply that
$\s(i)=\t(i)$. Summarizing we have proved that $v\det(x_i^j)\det(y_i^j)\in I$. On the other hand, from  
$\big(z\det(x_i^j)-1\big)\big(w\det(y_i^j)-1\big) \in I$ we obtain that 
\begin{align*}
v\big(z\det(x_i^j) - 1\big)\big(w\det(y_i^j) - 1\big) &= vzw\det(x_i^j)\det(y_i^j)-vz\det(x_i^j)-vw\det(y_i^j)+v
\\
& \equiv -v-v+v=-v \pmod{I},
\end{align*}
which finishes the proof.
\end{proof}

\begin{lemma} \label{idealsProduct}
    $\prod_{\s\in S_n}J_\s=I$
\end{lemma}
\begin{proof} 
Lemma \ref{emirp} (ii) tells us that the ideals $J_\s$ are pairwise coprime, which implies that $\prod_\s J_\s = \cap_\s J_\s$.
Moreover, from Lemma \ref{emirp} (i) we get that $I\subseteq \cap_\s J_s$. We claim that $\prod_\s J_\s\subset I$. In fact, 
write $J_s$ as in Lemma \ref{emirp} (iii), and consider the element $\prod_\s a_\s$, where each $a_\s\in \{x_i^j\}_{j\ne\s(i)}\cup \{y_i^j\}_{j\ne\s(i)}\cup \{z\det(x_i^j) - 1, w\det(y_i^j) - 1\}$ is a generator of $J_\s$. If either $a_\s = z\det(x_i^j) - 1$ or $a_\s =  w\det(y_i^j) - 1$ for some $\s$, then we are done since both of these elements are in $I$.
To finish, if $a_\s \in \{x_i^j\}_{j\ne\s(i)}\cup \{y_i^j\}_{j\ne\s(i)}$ for all $\s$, then the result follows from Lemma \ref{oetse}.
\end{proof}

Given a permutation $\s\in S_n$, the automorphism $A\to A$ induced by 
$x_i^j\mapsto x_i^{\s(j)}$, $y_i^j\mapsto y_i^{\s(j)}$, $z\mapsto (-1)^{\vert\s\vert}z$, 
$w\mapsto (-1)^{\vert\s\vert}w$
maps the ideal $J_1$ to $J_\s$ and so provides an isomorphism $\Omega_\s\colon A/J_\s\to A/J_1$.
We are in a position now to prove the main result of this section.

\begin{theorem} \label{TAUT}
Let $\A$ be an $n$-dimensional perfect evolution algebra over a field $\K$. The identity component of the affine group scheme of ternary automorphisms $\taffaut(\A)$ of $\A$ is isomorphic to the torus $\Gm^{2n}$. Moreover: $\taffaut(\A)\cong\Gm^{2n}\times \underline{S_n}$.  
\end{theorem}

\begin{proof}
From \cite[{\bf 1.4}, p. 7]{Jantzen} we have $\taffaut(\A)\cong V(I)$. Moreover, Lemma \ref{idealsProduct} tells us that 
$I=\prod_\s J_\s$, and so we have, scheme-theoretically, $V(I) = \cup_{\s\in S_n} V(J_\s)$. This 
implies that  $V(I)(D) = \cup_{\s\in S_n} V(J_\s)(D)$, for any domain $D\in\alg_\K$. The connected component of the unit is the affine group scheme 
$\taffaut_0(\A) = V(J_1)$, whose representing Hopf algebra is $\mathcal{H}:=K[x_i^k,y_i^k,z,w]/J_1$ and $J_1 = \Big(\{x_{ij}\}_{i\neq j}\cup\{y_{ij}\}_{i\neq j}\cup\{z\det(x_i^j)-1,w\det(y_i^j)-1\}_{i, j\in \Lambda}\Big)$ by Lemma \ref{emirp} (iii). Therefore:
\[
\mathcal{H}\cong K[x_{ii},y_{ii},z,w]/(z\prod_i x_{ii}-1,w\prod_i y_{ii}-1)\cong \otimes_i \left(K[x_{ii}^\pm]\otimes K[y_{ii}^\pm]\right)\cong K[x^\pm]^{\otimes 2n},
\]
which yields that $\taffaut_0(\A)=\Gm^{2n}$. Furthermore, $A/I = A/\prod_\s J_\s\cong \prod_\s A/J_\s\cong A/J_1\otimes \K^{n!}$, where the last isomorphism is
$(x_\s+I_\s)_\s\mapsto \sum_\s \Omega_\s(x_\s+J_\s)\otimes u_\s$. Here $\{u_\s\}$ denotes the canonical basis of $\K^{n!}$ modulo
the identification $\K^{n!}\cong \K^{S_n}$ so that for each $\s\in S_n$ the element $u_\s\colon S_n\to K$ is given by
$u_\s(\t)=\delta_{\s,\t}$ (Kronecker's delta). 
\end{proof}

\begin{cor} \label{TderPer}
The Lie algebra of ternary derivations of a perfect $n$-dimensional evolution algebra $\A$ is 
isomorphic to $\D_n(\K)\times\D_n(\K)$. 
\end{cor}

\begin{proof}
It follows from the fact that the Lie algebra of the group scheme $\taffaut_0(\A)$ is precisely the algebra of the ternary derivations of $\A$: $
\tder(\A) = \hbox{Lie}(\taffaut_0(\A))\cong\D_n(\K)\times\D_n(\K).$
\end{proof}

\begin{cor} \label{perfectcase}
The following hold for an $n$-dimensional perfect evolution algebra $\A$ with natural basis $\B:=\{e_i\}_1^n$.
\begin{itemize}
\item[\rm (i)] if $(f_1,f_2,f_3) \in \taut(\A)$, then $f_1$ is completely determined by $f_2$ and $f_3$. 
More precisely, for any $(\lambda_1,\ldots,\lambda_n), (\mu_1,\ldots,\mu_n)\in (\K^\times)^n,$ let $f_2, f_3\colon\A\to\A$ be the linear maps whose matrices relative to $\B$ are $\hbox{diag}(\lambda_1,\ldots,\lambda_n)$ and $\hbox{diag}(\mu_1,\ldots,\mu_n)$, respectively. Then defining 
$f_1\colon\A\to\A$ 
as the linear map such that $f_1(e_i^2)=\lambda_i\mu_i e_i^2$ for any $i$, the triple $(f_1,f_2,f_3)$ is a ternary
automorphism of $\A$ and any element in $\taut(\A)_0$ arises in this way. Ternary automorphisms that are not in the component
$\taut(\A)_0$ arise by applying a permutation to those of the component  $\taut(\A)_0$.
\item[\rm (ii)] if $(d_1,d_2,d_3) \in \tder(\A)$, then $d_1$ is completely determined by $d_2$ and $d_3$. 
More precisely, for any two $(\lambda_1,\ldots,\lambda_n), (\mu_1,\ldots,\mu_n)\in \K^n,$ let $d_2, d_3\colon\A\to\A$ be the linear maps 
whose matrices 
relative to $\B$ are $\hbox{diag}(\lambda_1,\ldots,\lambda_n)$ and $\hbox{diag}(\mu_1,\ldots,\mu_n)$, respectively. Define
$d_1\colon\A\to\A$ as the unique
linear map such that $d_1(e_i^2)=(\lambda_i+\mu_i) e_i^2$ for any $i$. Then $(d_1,d_2,d_3)\in\tder(\A)$ and 
any ternary derivation of $\A$ arises in this way.
\end{itemize}
\end{cor}


\section{Ternary derivations} \label{sec4}


In this section, we go beyond the perfect case. We determine necessary and sufficient conditions for a triple of linear maps $(d_1, d_2, d_3)$ 
on an evolution algebra $\A$ to be a ternary derivation of $\A$. If $\A$ is perfect, 
we provide an expression for $d_1$ in terms of the other components.
Otherwise, if $\M_\B$ is singular, then the entries of $\text{diag}(d_2 + d_3)$ are no longer independent from one another, 
as the eigenspaces for the corresponding column vectors overlap. 
We divide our study into two steps: we first determine the necessary and sufficient condition 
that $d_2$ and $d_3$ must satisfy; secondly, given $d_2$ and $d_3$ we find a formula for $d_1$.

\smallskip

In what follows, we assume that $\A$ is a finite-dimensional evolution algebra with natural basis $\B$ and structure matrix $\M_\B$ as in \S2.1. For a triple of linear maps $(d_1, d_2, d_3)$ of $\A$ we write $\, d_{\ell}(e_i) = \sum^n_{k = 1} d^{(\ell)}_{ki}e_k$, for $\, \ell = 1, 2, 3.$

\begin{prop} \label{iffCondition}
Let $\A$ be a finite-dimensional evolution algebra.  
A triple of linear maps $(d_1, d_2, d_3)$ on $\A$ is a ternary derivation on $\A$ if and only if
it satisfies the identities: 
\allowdisplaybreaks
\begin{align}
& w_{ki}d^{(2)}_{ij} + w_{kj}d^{(3)}_{ji} = 0, \, \, \mbox{ for all} \, \, k, i, j \in \Lambda \, \mbox{ with } i \neq j, \tag{Td1} \label{Td1}
\\
& \sum^n_{k = 1} w_{ki}d^{(1)}_{jk} = \big(d^{(2)}_{ii} + d^{(3)}_{ii}\big)w_{ji}, \, \, \mbox{ for all } \, \, j.  \tag{Td2} \label{Td2}
\end{align}
\end{prop}

\begin{proof}
Assume first that $(d_1, d_2, d_3)$ is a ternary derivation on $\mathcal{A}$ and take $e_i, e_j \in \mathcal{B}$ with $i \neq j$. Using \eqref{E2} a couple of  times, \eqref{Td} and \eqref{E1} we obtain that 
\allowdisplaybreaks
\begin{align*}
 0 & = d_1(e_i e_j) = e_i d_2(e_j) + d_3(e_i) e_j =  
 e_i \sum_k d^{(2)}_{kj} e_k + e_j \sum_k d^{(3)}_{ki} e_k
 = d^{(2)}_{ij}e_i^2 + d^{(3)}_{ji}e_j^2 
 \\
 & = d^{(2)}_{ij}\sum_k w_{ki}e_k + d^{(3)}_{ji}\sum_k w_{kj}e_k = \sum_k (d^{(2)}_{ij} w_{ki} + d^{(3)}_{ji}w_{kj})e_k,
\end{align*}
and \eqref{Td1} follows. To prove \eqref{Td2} we apply \eqref{Td} and (E1) and equate the resulting equations:
\begin{align*}
d_1(e_i^2) & =  e_i d_2(e_i) + e_i d_3(e_i) = e_i \sum_k d^{(2)}_{ki} e_k + e_i \sum_k d^{(3)}_{ki} e_k = 
d^{(2)}_{ii}e_i^2 + d^{(3)}_{ii}e_i^2 
\\
& \stackrel{\rm (E2)}{=} \sum_k \big(d^{(2)}_{ii} + d^{(3)}_{ii}\big)w_{ki}e_k,
\\
d_1(e_i^2) & = d_1\Big(\sum_j w_{ji}e_j\Big) = \sum_j w_{ji} d_1(e_j) = \sum_k \Big(\sum_j w_{ji}d^{(1)}_{kj}\Big)e_k.
\end{align*}
Clearly, \eqref{Td2} holds. Conversely, assume that $(d_1, d_2, d_3)$ satisfies both \eqref{Td1} and \eqref{Td2}. Take $x, y \in \A$ and write $x = \sum_{k\in \Lambda} \lambda_k e_k$ and $y = \sum_{k\in\Lambda}\eta_k e_k$.
 
\allowdisplaybreaks
\begin{align*}
d_1(xy) &=  d_1\left(\Big(\sum_{k\in \Lambda} \lambda_k e_k\Big)\Big(\sum_{k\in\Lambda}\eta_k e_k\Big)\right) = d_1\Big(\sum_{k\in \Lambda} \lambda_k\eta_k e_k^2\Big) \stackrel{{\rm(E1)}}{=} 
\sum_{k\in \Lambda} \lambda_k\eta_k d_1\Big(\sum_{j\in\Lambda}w_{jk}e_j\Big) \\ &=  
\sum_{k\in \Lambda} \lambda_k\eta_k \sum_{j\in\Lambda}w_{jk} d_1(e_j) = 
\sum_{k \in \Lambda} \lambda_k\eta_k \sum_{j\in\Lambda}w_{jk} \sum_{i\in\Lambda} d^{(1)}_{ji}e_i =   
\sum_{k\in \Lambda} \lambda_k\eta_k \sum_{i\in\Lambda} \sum_{j\in\Lambda}w_{jk} d^{(1)}_{ji}e_i\\ &= 
\sum_{k\in \Lambda} \lambda_k\eta_k \sum_{i\in\Lambda}\Big(\sum_{j\in\Lambda}w_{jk} d^{(1)}_{ij}\Big)e_i \stackrel{\eqref{Td2}}{=} \sum_{k\in \Lambda} \lambda_k\eta_k \sum_{i\in\Lambda} \Big(\big(d^{(2)}_{kk} + d^{(3)}_{kk}\big)w_{ik}\Big)e_i \\& \stackrel{{\rm(E1)}}{=} 
\sum_{k\in \Lambda} \lambda_k\eta_k \big(d^{(2)}_{kk} + d^{(3)}_{kk}\big) e_k^2.
\end{align*}
On the other hand, we have that 
\begin{align*}
d_2(x)y + d_3(y)x &=  d_2\Big(\sum_{i\in\Lambda}\lambda_i e_i\Big) \Big(\sum_{j\in\Lambda}\eta_j e_j\Big) + \Big(\sum_{i\in\Lambda}\lambda_i e_i\Big) d_3\Big(\sum_{j\in\Lambda}\eta_j e_j\Big) \\ 
&= \Big(\sum_{i\in\Lambda}\lambda_i d_2(e_i)\Big)\Big(\sum_{j\in\Lambda}\eta_j e_j\Big) + \Big(\sum_{i\in\Lambda}\lambda_i e_i\Big)\Big(\sum_{j\in\Lambda}\eta_j d_3(e_j)\Big)\\ 
&=  \Big(\sum_{i\in\Lambda}\lambda_i \sum_{j\in\Lambda} d^{(2)}_{ji} e_j\Big)\Big(\sum_{j\in\Lambda}\eta_j e_j\Big) + \Big(\sum_{i\in\Lambda}\lambda_i e_i\Big)\Big(\sum_{j\in\Lambda}\eta_j\sum_{i\in\Lambda} d^{(3)}_{ij}e_i\Big)\\
&=  \Big(\sum_{j\in\Lambda}\Big(\sum_{i\in\Lambda}\lambda_i d^{(2)}_{ji}\Big)e_j\Big)\Big(\sum_{j\in\Lambda}\eta_j e_j\Big) + \Big(\sum_{i\in\Lambda}\lambda_i e_i\Big)\Big(\sum_{i\in\Lambda}\Big(\sum_{j\in\Lambda}\eta_j d^{(3)}_{ij}\Big)e_i\Big)\\
& =  \sum_{j\in\Lambda}\Big(\sum_{i\in\Lambda}\lambda_i d^{(2)}_{ji}\Big)\eta_j e_j^2 + \sum_{i\in\Lambda}\lambda_i\Big(\sum_{j\in\Lambda}\eta_j d^{(3)}_{ij}\Big)e_i^2 \\
&= \sum_{i\in\Lambda}\sum_{j\in\Lambda}\lambda_i\eta_j d^{(2)}_{ji} e_j^2 + \sum_{i\in\Lambda}\sum_{j\in\Lambda}\lambda_i\eta_j d^{(3)}_{ij}e_i^2\\
&\stackrel{\eqref{E1}}{=} \sum_{i\in\Lambda}\sum_{j\in\Lambda}\lambda_i\eta_j d^{(2)}_{ji}\sum_{k\in\Lambda}w_{kj}e_k + \sum_{i\in\Lambda}\sum_{j\in\Lambda}\lambda_i\eta_j d^{(3)}_{ij}\sum_{k\in\Lambda}w_{ki}e_k \\
& = \sum_{k\in\Lambda}\sum_{i\in\Lambda}\sum_{j\in\Lambda}\lambda_i\eta_j \Big(d^{(2)}_{ji}w_{kj} + d^{(3)}_{ij}w_{ki}\Big) e_k
\stackrel{\eqref{Td1}}{=} \sum_{k\in\Lambda}\sum_{i\in\Lambda} \lambda_i\eta_i \Big(\big(d^{(2)}_{ii}+d^{(3)}_{ii}\big)w_{ki}\Big)e_k \\
& = \sum_{i\in\Lambda} \lambda_i\eta_i \big(d^{(2)}_{ii} + d^{(3)}_{ii}\big) e_i^2,
\end{align*}
and the result follows. 
\end{proof}

Notice that we can express equations \eqref{Td1} and \eqref{Td2} as linear systems of equations. To do so, 
we write $[\M_\B]_\ell$ to refer to the $\ell^{\text{th}}$ column of $\M_\B$.  
Given $i, j \in \Lambda$ with $i \neq j$, from \eqref{Td1} we obtain that the following equivalent system:
\begin{equation} \label{system1} 
\begin{pmatrix}
[\M_\B]_i & [\M_\B]_j 
\end{pmatrix}
\begin{pmatrix}
d^{(2)}_{ij} 
\\
 d^{(3)}_{ji}
\end{pmatrix} = 0.\tag{S1}
\end{equation} 
Similarly, from \eqref{Td2} we obtain the equivalent system given below. 
\begin{equation} \label{system2} 
\Big(d^{(1)}_{j1} \, \ldots \, d^{(1)}_{jn}\Big) 
\M_\B = 
\begin{pmatrix}
 w_{j1} \big(d^{(2)}_{11}+d^{(3)}_{11}\big)
\\ 
\vdots 
\\
w_{jn} \big(d^{(2)}_{nn} + d^{(3)}_{nn}\big)
\end{pmatrix} \tag{S2},
\end{equation}
or equivalently, 
\begin{equation} \label{system3} 
d_1 \M_\B  = \M_\B \mathrm{diag}(d_2 + d_3). \tag{S3}
\end{equation} 
From here, we deduce that $d_1$ must be a matrix with eigenvectors being the columns of $\M_\B$, with the entries on the diagonal of $d_2 + d_3$ being the corresponding eigenvalues.

In what follows, we do not require $\A$ to be perfect; that is, its structure matrix $\M_\B$ may be singular. 
We assume that $e_1^2, \ldots, e_r^2$ are independent. For $i \in \{1, \ldots ,n-r\},$ we define the matrix 
$C = (c_{ki}) \in M_{r \times n - r}(\K)$ by $\, e_{r+i}^2 = \sum_{k=1}^r c_{ki}e_k^2.$ 

We first investigate $d_2$ and $d_3$. The proof of the following result is straightforward and we leave 
it to the reader.

\begin{prop} \label{d2d3}
Let $d_2$ and $d_3$ be linear maps on an $n$-dimensional evolution algebra $\A$. Then \eqref{system1} holds if and only if the following condition hold for all $i,j\in\{1,\ldots,n\}$.
    \begin{enumerate}
        \item[\rm (i)] If $e_i^2 = ce_j^2$ for $c\neq0$, then $d^{(3)}_{ij},d^{(3)}_{ji}$ are free, with  $d^{(2)}_{ij} = -cd^{(3)}_{ji}$ and $d^{(2)}_{ji}=-c^{-1}d^{(3)}_{ij}$.
        \item[\rm (ii)]  If $e_i^2=0$ and $e_j^2\neq0$, then $d^{(2)}_{ji}$ and $d^{(2)}_{ij}$ are free and $d^{(2)}_{ij}=d^{(2)}_{ij}=0$.
        \item[\rm (iii)]  If $e_i^2=e_j^2=0$, then $d^{(2)}_{ij},d^{(2)}_{ji},d^{(3)}_{ij},$ and $d^{(3)}_{ji}$ are free.
        \item[\rm (iv)]  If $e_i^2$ and $e_j^2$ are independent, then $d^{(2)}_{ij}=d^{(2)}_{ji}=d^{(3)}_{ij}=d^{(2)}_{ji}=0$.
    \end{enumerate}
\end{prop}

\begin{cor} \label{iffTD}
A triple of linear maps $(d_1, d_2, d_3)$ on a perfect finite-dimensional evolution algebra $\A$ is a ternary derivation of $\A$ if and only if the following conditions hold.
\begin{enumerate}
\item[\rm (i)] The matrices of $d_2$ and $d_3$ (with respect to $\B$) are diagonal;
\item[\rm (ii)] $d_1 = \M_\B(d_2 + d_3)\M_\B^{-1}.$
\end{enumerate}
\end{cor}

\begin{proof}
Assume that (i) and (ii) hold. From (i) we get that $d^{(2)}_{ji} = d^{(3)}_{ij} = 0$ for all $i, j \in \Lambda$ with $i \neq j$, and \eqref{Td1} holds. Moreover, (ii) implies that \eqref{system3} has a solution, which tells us that \eqref{Td2} is satisfied. The converse follows from Proposition \ref{d2d3} (iv) and \eqref{system3}.
\end{proof}
 
If $\M_\B$ is nonsingular (that is, $\A$ is perfect), then we have proved that the resulting matrix for $d_1$ is uniquely determined by $\M_\B$ and the chosen eigenvalues. However, this fails to be the case when $\M_\B$ is singular, as its columns no longer span the underlying algebra. In such a case, there are more possible matrices with the same eigenvector-eigenvalue combinations. To be able to find a closed form solution for $d_1$, we make use here of the concept of a generalised inverse for a square matrix.

\begin{define}
A matrix $G$ is called a {\bf generalised inverse} of a square matrix $A$ if the equation $AGA = A$ holds.
\end{define}

Recall that an element $x$ of a ring $R$ is said to be {\bf von Neumann regular} if there is some $y\in R$
such that $xyx = x$. A ring whose all elements are von Neumann regular is called a {\bf von Neumann regular ring}.
It is well known that the matrix ring $M_n(\K)$ is von Neumann regular. 
Consequently, all square matrices have a generalised inverse. At this point, it is worth mentioning that there might be more than one solution for the generalised inverse of a matrix. We select an appropriate choice for $G$ in the next result.

\begin{lemma} \label{gmatrix}
The following hold for an evolution algebra $\A$ as before.
\begin{enumerate}
\item[\rm (i)] The 
reduced row-echelon form of the structure matrix $\M_\B$ of $\A$ is the block matrix $\M_\B':= \begin{pmatrix}
    I_r & C\\
    O & O
\end{pmatrix},$ where $C \in M_{r, \, n - r}(\K)$.

\item[\rm (ii)] Any matrix of the form $\begin{pmatrix}
    I_r & O\\
    O & K
\end{pmatrix},$ where $K$ is an arbitrary matrix, is a generalised inverse of $\M_\B'$. 
\item[\rm (iii)] The invertible matrix $G$ such that $G \M_\B = \M_\B'$ is a generalised inverse of $\M_\B$.
\end{enumerate}
\end{lemma}

\begin{proof}
(i) and (ii) are straightforward calculations and we leave it to the reader. 

\smallskip

\noindent (iii) From (ii) we know that $I_n$ is a generalised inverse of $\M_\B'$. Thus: 
\[
\M_\B G \M_\B = G^{-1} \M_\B' I_n \M_\B' = G^{-1}\M_\B' = \M_\B,
\]
finishing the proof.
\end{proof}

\begin{lemma} \label{partialsol}
Let $d_2, d_3$ be linear maps on $\A$ and $G$ a generalised inverse of $\M_\B$.
Any solution of \eqref{system3} is of the form 
$\M_\B\mathrm{diag}(d_2+d_3)G + A$, where $A$ is such that $A\M_\B = 0$.
\end{lemma}

\begin{proof}
Assume that $X$ is a solution to \eqref{system3}. Then $\M_\B \text{diag}(d_2+d_3) = X \M_\B = X \M_\B G \M_\B = \M_\B\mathrm{diag}(d_2+d_3)G \M_\B$, since $G$ is a generalised inverse of $\M_\B$. From here, we obtain that 
$(X- \M_\B \text{diag}(d_2+d_3)G)\M_\B =  X \M_\B - X \M_\B = 0.$
\end{proof}

Given two linear maps $d_2$ and $d_3$ on $\A$, in the previous result, we determined a general solution of \eqref{system3}. It remains to find a necessary and sufficient condition to guarantee the existence of a solution.

\begin{prop} \label{d1solution}
Let $d_2, d_3$ be linear maps on $\A$ and $\lambda_k := d^{(2)}_{kk}+d^{(3)}_{kk}$ for $k\in\{1,\ldots,n\}$. Then  \eqref{system3} has a solution for $d_1$ if and only if $c_{ji}\neq 0$ implies $\lambda_j = \lambda_{r+i}$, for all $i\in\{1,\ldots,n-r\}$ and $j\in\{1,\ldots,r\}$.
\end{prop}

\begin{proof}
Take $G$ as in Lemma \ref{gmatrix} (iii). Suppose first that \eqref{system3} holds and $c_{ji}\neq 0$ for some $i\in\{1,\ldots,n-r\}$ and $j\in\{1,\ldots,r\}$. Applying Lemma \ref{partialsol} we get that $d_1 := \M_\B\mathrm{diag}(d_2+d_3)G$ is a solution of \eqref{system3}. Consider $e_{r+i}^2 = \sum_{k=1}^r c_{ki}e_k^2$. By hypothesis, we have that $\lambda_{r+i}e_{r + i}^2 = d_1(e_{r+i}^2) = \sum_{k=1}^r c_{ki}d_1(e_k^2)$. Thus, $\sum_{k=1}^r c_{ki}\lambda_{r+i}e_k^2 = \sum_{k=1}^r c_{ki} \lambda_k e_k^2,$ since  
$d_1(e_j^2) = \lambda_j e_j^2$ for all $j = 1, \ldots, r$. From here, we derive that $c_{ji}\lambda_{r+i} = c_{ji}\lambda_j$, which implies that $\lambda_j = \lambda_{r+i},$ since we are assuming that $c_{ji}\neq0$.

Conversely, suppose $\lambda_j=\lambda_{r+i}$ whenever $c_{ji}\neq0$ for all $i\in\{1,\ldots,n-r\}$ and $j\in\{1,\ldots,r\}$. Then $\sum_{j=1}^r (\lambda_{i+r} - \lambda_j)c_{ji}w_{kj}=0$ for all $i\in\{1,\ldots,n-r\}$ and $k\in\{1,\ldots,n\}$, which implies that $-\sum_{j=1}^r \lambda_jc_{ji}w_{kj}+ \lambda_{i+r}w_{k,i+r} = 0$. From this, we get that \[\M_\B\text{diag}(d_2+d_3)(I - G\M_\B)= 0,\] so that $\M_\B\text{diag}(d_2+d_3) = \M_\B\text{diag}(d_2+d_3)G\M_\B$. This shows that \eqref{system3} has a solution, finishing the proof.
\end{proof}

We close the section with a precise description of the ternary derivations of an arbitrary finite dimensional evolution algebra. The proof follows from the previous results and we omit it. 

\begin{theorem} \label{charDerTarb}
Let $\A$ be an $n$-dimensional evolution algebra with structure matrix $\M_\B$ of rank $r$. Assume that $e^2_1, \ldots, e^2_r$ are independent and let $e_{r+i}^2 = \sum_{k=1}^r c_{ki}e_k^2$ for all $i\in\{1,\ldots,n-r\}$. 
A triple $(d_1, d_2, d_3)$ of linear maps on $\A$ is a ternary derivation of $\A$ if and only if the following conditions hold.
\begin{enumerate}
\item[\rm (i)] $d_1:= \M_\B\mathrm{diag}(d_2+d_3)G + A$, where $A$ is such that $A\M_\B = 0$ and $G$ is a generalised inverse of $\M_\B$;
\item[\rm (ii)] $c_{ji}\neq 0$ implies $d^{(2)}_{jj}+d^{(3)}_{jj} =  d^{(2)}_{r + i}+d^{(3)}_{r + i}$, for all $i\in\{1,\ldots,n-r\}$ and $j\in\{1,\ldots,r\}$; 
\item[\rm (iii)] If $e_i^2 = ce_j^2$ for $c\neq0$, then $d^{(3)}_{ij},d^{(3)}_{ji}$ are free, with  $d^{(2)}_{ij} = -cd^{(3)}_{ji}$ and $d^{(2)}_{ji}=-c^{-1}d^{(3)}_{ij}$.
\item[\rm (iv)]  If $e_i^2=0$ and $e_j^2\neq0$, then $d^{(2)}_{ji}$ and $d^{(2)}_{ij}$ are free and $d^{(2)}_{ij}=d^{(2)}_{ij}=0$.
\item[\rm (v)]  If $e_i^2=e_j^2=0$, then $d^{(2)}_{ij},d^{(2)}_{ji},d^{(3)}_{ij},$ and $d^{(3)}_{ji}$ are free.
\item[\rm (vi)]  If $e_i^2$ and $e_j^2$ are independent, then $d^{(2)}_{ij}=d^{(2)}_{ji}=d^{(3)}_{ij}=d^{(2)}_{ji}=0$.
\end{enumerate} 
\end{theorem}

\section{Ternary derivations on 2-dimensional evolution algebras} \label{case2}

In this section, we compute all the ternary derivations of 2-dimensional evolution algebras. A complete classification of evolution algebras of dimension 2 was provided in \cite{squares}, where the authors also investigated derivations and automorphisms of 2-dimensional evolution algebras (see \cite[Table 3]{squares}). They found 9 different isomorphism classes of 2-dimensional evolution algebra (see \cite[Lemma 3.1 and Table 4]{squares}). For completeness, we recall here the structure matrices of the representatives of these 9 isomorphism classes:

\smallskip 

\begin{longtable} [c] { | c | l| }
    \hline
     {\bf Algebra} & {\bf Multiplication table} \\
    \hline
    \endfirsthead
    \hline
    {\bf Algebra} & {\bf Multiplication table} \\
    \hline
    \endhead
    \hline
    \endfoot
    \hline
    \endlastfoot
     $\A_0$ &  
     $\begin{pmatrix}
     0 & 0 \\
     0 & 0
     \end{pmatrix}$   
     \\        
    \hline
     $\A_1$ &  
     $\begin{pmatrix}
     1 & 0 \\
     0 & 1
     \end{pmatrix}$
     \\
    \hline 
     $\A_{2, \alpha}$ &  
     $\begin{pmatrix}
     0 & \alpha \\
     1 & 0
     \end{pmatrix}$, \quad $\alpha \in \K^\times$
     \\
    \hline
     $\A_{3, \alpha}$ &  
     $\begin{pmatrix}
     1 & \alpha \\
     0 & 1
     \end{pmatrix}$, \quad $\alpha \in \K^\times$
     \\
    \hline
     $\A_{4, \alpha}$ &  
     $\begin{pmatrix}
     0 & 1 \\
     \alpha & 1
     \end{pmatrix}$, \quad $\alpha \in \K^\times$
     \\
     \hline
     $\A_{5, \alpha, \beta}$ &  
     $\begin{pmatrix}
     1 & \alpha \\
     \beta & 1
     \end{pmatrix}$, \quad $\alpha, \beta \in \K^\times$, $\alpha \beta \neq 1$
     \\
     \hline
     $\A_5$ &  
     $\begin{pmatrix}
     1 & -1 \\
     -1 & 1
     \end{pmatrix}$
     \\
     \hline
     $\A_6$ &  
     $\begin{pmatrix}
     0 & 1 \\
     0 & 0
     \end{pmatrix}$
     \\
     \hline
     $\A_7$ &  
     $\begin{pmatrix}
     1 & 0 \\
     0 & 0
     \end{pmatrix}$
     \\
     \hline
     $\A_{8, \alpha}$ &  
     $\begin{pmatrix}
     1 & \alpha \\
     0 & 0
     \end{pmatrix}$, \quad $\alpha \in \K^\times$
     \\
     \hline
    \caption{\label{SM} {\small Structure matrices of 2-dim evolution algebras.}}
    \end{longtable}
    
\smallskip

For a triple $(d_1, d_2, d_3)$ of linear maps of $\A$ (for $\A$ one of the 2-dimensional evolution algebras given in Table \ref{SM}) we write 
\[
d_1  = 
\begin{pmatrix}
x_{11} & x_{12}
\\
x_{21} & x_{22}
\end{pmatrix}, 
\quad 
d_2  = 
\begin{pmatrix}
y_{11} & y_{12}
\\
y_{21} & y_{22}
\end{pmatrix},
\quad 
d_3  = 
\begin{pmatrix}
z_{11} & z_{12}
\\
z_{21} & z_{22}
\end{pmatrix},
\]
where $x_{ij}, y_{ij}, z_{ij} \in \K$. Let $\lambda:= y_{11} + z_{11}$ and $\mu:= y_{22} + z_{22}$.

\subsection{Perfect case}
Notice that the algebras $\A_1$, $\A_{2, \alpha}$, $\A_{3, \alpha}$, 
$\A_{4, \alpha}$, $\A_{5, \alpha, \beta}$ are all perfect, so we can apply Corollary \ref{iffTD} to determine the space of their ternary derivations. 
An application of Corollary \ref{iffTD} (i) tells us that the matrices of $d_2$ and $d_3$ are diagonal, so $y_{12} = y_{21} = z_{12} = z_{21} = 0$ for $\A \in \{\A_1, \A_{2, \alpha}, \A_{3, \alpha}, 
\A_{4, \alpha}, \A_{5, \alpha, \beta}\}$. Moreover, from Corollary \ref{iffTD} (i) we obtain the following matrices for $d_1$: 
\begin{itemize}
\item[-] if $\A = \A_1$, then $d_1 = d_2 + d_3 = \mathrm{diag}(\lambda, \mu);$
\item[-] if $\A = \A_{2, \alpha}$, then $d_1 = \mathrm{diag}(\mu, \lambda);$
\item[-] if $\A = \A_{3, \alpha}$, then $d_1 = 
\begin{pmatrix}
\lambda & \alpha(\mu - \lambda)
\\
0 & \mu
\end{pmatrix};$
\item[-] if $\A = \A_{4, \alpha}$, then $d_1 = 
\begin{pmatrix}
\mu & 0
\\
\mu - \lambda & \lambda
\end{pmatrix};$
\item[-] if $\A = \A_{5, \alpha, \beta}$, then $d_1 = \frac{1}{\gamma - 1}
\begin{pmatrix}
\gamma\mu - \lambda & \alpha(\lambda - \mu)
\\
\beta(\lambda + \mu) & \gamma\lambda - \mu
\end{pmatrix},$ for \, $\gamma = \alpha \beta$.
\end{itemize}

\subsection{Non-perfect 2-dimensional evolution algebras}
Notice that any triple of linear maps on $\A_0$ will be a ternary derivation of $\A_0$, so there is nothing to compute in this case. It remains to compute the ternary derivations of $\A_5$, $\A_6$, $\A_7$ and $\A_{8, \alpha}$. As before, we write 
$\begin{pmatrix}
w_{11} & w_{12} \\ 
w_{21} & w_{22}
\end{pmatrix}$ to denote the structure matrix. From \eqref{Td1} and \eqref{Td2}, or equivalently, from \eqref{system1} and \eqref{system2}, we obtain the following system of linear equations:
\begin{align*}
w_{11}y_{12} + w_{12}z_{21} & = 0
\\
w_{21}y_{12} + w_{22}z_{21} & = 0
\\
w_{12}y_{21} + w_{11}z_{12} & = 0
\\
w_{22}y_{21} + w_{21}z_{12} & = 0
\\
w_{11}x_{11} + w_{21}x_{12} &= w_{11} \lambda 
\\
w_{12}x_{11} + w_{22}x_{12} &= w_{12} \mu
\\
w_{11}x_{21} + w_{21}x_{22} &= w_{21} \lambda 
\\
w_{12}x_{21} + w_{22}x_{22} &= w_{22} \mu
\end{align*}
It is a straightforward calculation to solve the resulting systems for $\A_5$, $\A_6$, $\A_7$ and $\A_{8, \alpha}$; we leave this to the reader. For $\A_5$, we obtain that $z_{12} = y_{21}$, $z_{21} = y_{12}$, $\lambda = y_{11} + z_{11} = y_{22} + z_{22} = \mu$, $x_{11} = \lambda + x_{12}$ and $x_{22} = \lambda + x_{21}$. In the case of $\A_6$, the matrices of $d_1$, $d_2$ and $d_3$ are all upper triangular $x_{11} = \mu$. For $\A_7$, we get that the matrices of $d_2$ and $d_3$ are lower triangular, while the matrix of $d_1$ is upper triangular with $x_{11} = \lambda$. Lastly, for $\A_{8, \alpha}$, we obtain that $z_{12} = -\alpha y_{21}$, $z_{21} = -\frac 1{\alpha} y_{12}$, $x_{21} = 0$ and $x_{11} = \lambda$.  


\section*{Acknowledgements}
The first author was supported by Junta de Andaluc\'ia  through projects FQM-336, UMA18-FEDERJA-119, and PAIDI project P20\_01391, and by the Spanish Ministerio de Ciencia e Innovaci\'on through projects  PID2019-104236GB-I00 and PID2020-118452GB-I00, all of them with FEDER funds. The second author was supported by a URC Postdoctoral Research Fellowships from the University of Cape Town. The third author was supported by URC grant from the University of Cape Town.


\end{document}